\documentclass{amsart}

\usepackage{graphicx}

\newtheorem{theorem}{Theorem}
\newtheorem{lemma}[theorem]{Lemma}

\begin{document}

\title[Cell decompositions of Teichm\"uller spaces]{Cell decompositions of Teichm\"uller spaces of surfaces with boundary}

\author{Ren Guo}

\address{School of Mathematics, University of Minnesota, Minneapolis, MN, 55455}

\email{guoxx170@math.umn.edu}

\author{Feng Luo}

\address{Department of Mathematics, Rutgers University, Piscataway, NJ, 08854}

\email{fluo@math.rutgers.edu}

\thanks{The work of the second named author is supported in part by the NSF}



\begin{abstract}
A family of coordinates $\psi_h$ for the Teichm\"uller space of a
compact surface with boundary was introduced in \cite{l2}. In the
work \cite{m1}, Mondello showed that the coordinate $\psi_0$ can
be used to produce a natural cell decomposition of the
Teichm\"uller space invariant under the action of the mapping
class group. In this paper, we show that the similar result also
works for all other coordinate $\psi_h$ for any $h \geq 0$.
\end{abstract}

\maketitle

\section{Introduction}

In this note, we show that each of the  coordinate $\psi_h$ ($h\geq 0$)
introduced in \cite{l2} can be used to produce a natural cell
decomposition of the Teichm\"uller space of a compact surface with non-empty boundary and
negative Euler characteristic. We will show that the
underlying point sets of the cells are the same as the one
obtained in the previous work of Ushijima \cite{u}, Hazel \cite{ha}, Mondello \cite{m1}.
However, the coordinates $\psi_h$ for $h \geq 0$ introduce
different attaching maps for the cell decomposition. In the sequel, unless mentioned otherwise, we will always assume that the surface $S$ is compact with non-empty boundary so that the Euler characteristic of $S$ is negative.

\subsection{The arc complex}

We begin with a brief recall of the related concepts. An \textit{essential arc} $a$ in $S$ is an embedded arc with boundary in $\partial S$ so that $a$ is not
homotopic into $\partial S$ relative to $\partial S$.  The arc
complex $A(S)$ of the surface, introduced by J. Harer \cite{har1}, is
the simplicial complex so that each vertex is the homotopy
class $[a]$ of an essential arc $a$, and its simplex
is a collection of distinct vertices $[a_1],\ldots,[a_k]$ such that $a_i \cap a_j = \emptyset$ for all $i \neq j$.
For instance, the isotopy class of an ideal triangulation corresponds to a simplex
 of maximal dimension in $A(S)$. The non-fillable subcomplex $A_\infty(S)$ of $A(S)$ consists
of those simplexes $([a_1],\ldots,[a_k])$ such that one component of
$S-\cup_{i=1}^{k}a_i$ is not simply connected. The simplices in $A(S)-A_\infty(S)$ are called fillable. The underlying space of $A(S)-A_\infty(S)$ is denoted by $|A(S)-A_\infty(S)|$.

\subsection{The Teichm\"uller space}

It is well-known that there are
hyperbolic metrics with totally geodesic boundary on the surface $S$. Two hyperbolic metrics with geodesic boundary on $S$ are called isotopic if there is an isometry
isotopic to the identity between them. The space of all isotopy classes
of hyperbolic metrics with geodesic boundary on $S$ is called the Teichm\"uller space of the surface $S$, denoted by $Teich(S)$. Topologically, $Teich(S)$ is homeomorphic to a ball of dimension $6g-6+3n$ where $g$ is the genus and $n>0$ is the number of boundary components of $S$.

\begin{theorem}[Ushijima \cite{u}, Hazel \cite{ha}, Mondello \cite{m1}]\label{thm:0} There is a natural cell decomposition of the Teichm\"uller space $Teich(S)$ invariant under the action of the mapping class group.
\end{theorem}

Ushijima \cite{u} proved this theorem by following Penner's convex hall construction \cite{p}. Following Bowditch-Epstein's approach \cite{be}, Hazel \cite{ha} obtained a cell decomposition of the Teichm\"uller space of surfaces with geodesic boundary and fixed boundary lengths. In \cite{l1}, the second named author introduced $\psi_0$-coordinate to parameterize the Teichm\"uller space $Teich(S)$ of a surface $S$ with a fixed ideal triangulation. Mondello \cite{m1} pointed out that the $\psi_0$-coordinate produces a natural cell decomposition of $Teich(S)$.

In \cite{l2}, for each real number $h$, the second named author introduced the $\psi_h$-coordinates to parameterize $Teich(S)$ of a surface $S$ with a fixed ideal triangulation. The $\psi_0$-coordinate is a special case of the $\psi_h$-coordinates.

The main theorem of the paper is the following. 

\begin{theorem}\label{thm:main}
Suppose $S$ is a compact surface with non-empty boundary and negative Euler characteristic. For each $h\geq 0$, there is a homeomorphism
$$\Pi_h: Teich(S)\to |A(S)-A_\infty(S)|\times \mathbb{R}_{>0}$$
equivariant under the action of the mapping class group so that the restriction of $\Pi_h$ on each simplex of maximal dimension is given by the $\psi_h$-coordinate. In
particular, this map produces a natural cell decomposition of the moduli space
of surfaces with boundary.
\end{theorem}

We will show that the underlying cell-structures for various $h'$s are the same.

\subsection{Related results}

For a punctured surface $S$ with weights on each puncture, the classical Teichm\"uller space of $S$
admits cell decompositions. This was first proved by Harer \cite{har1}
and Thurston (unpublished) using Strebel's work on quadratic
differentials and flat cone metrics. The corresponding result in
the context of hyperbolic geometry was proved by Bowditch-Epstein \cite{be} and
Penner \cite{p} using complete hyperbolic metrics of finite area on $S$ so
that each cusp has an assigned horocycle. The constructions in
\cite{be} and \cite{p} are more geometrically oriented. Indeed, the construction of spines and
Delaunay decompositions based on a given set of points and
horocycles are used in \cite{be}. Our approach is the same as that of \cite{be}
using Delaunay decompositions. The existence of such Delaunay
decompositions for compact hyperbolic manifolds with geodesic
boundary was established in the work of Kojima \cite{k} for 3-manifolds. However, the same method of proof in \cite{k} also works for compact hyperbolic surfaces. Our main
observation in this paper is that those $\psi_h$-coordinates
introduced in \cite{l2} capture the Delaunay condition well.

\subsection{Plan of the paper}

In section 2, we recall the definition and properties of $\psi_h$-coordinates
 which will be used in the proof of Theorem \ref{thm:main}.
 In section 3, we prove a simple lemma which clarifies the
 geometric meaning of $\psi_h$-coordinates. In section 4,
 we review the Delaunay decomposition associated to a hyperbolic
 metric following Bowditch-Epstein \cite{be} and Kojima \cite{k}. Theorem \ref{thm:main} is proved in section 5.

\section{$\psi_h$-coordinates}

An ideal triangulated compact surface with boundary $(S, T )$ is obtained by
removing a small open regular neighborhood of the vertices of a triangulation of
a closed surface. The edges of an ideal triangulation $T$ correspond bijectively
to the edges of the triangulation of the closed surface. Given a hyperbolic
metric $d$ with geodesic boundary on an ideal triangulated surface $(S, T)$, there
is a unique geometric ideal triangulation $T^*$ isotopic to $T$ so that all edges are
geodesics orthogonal to the boundary. The edges in $T^*$ decompose the surface
into hyperbolic right-angled hexagons.

\begin{figure}[htbp]
\begin{center}
\includegraphics[scale=.3]{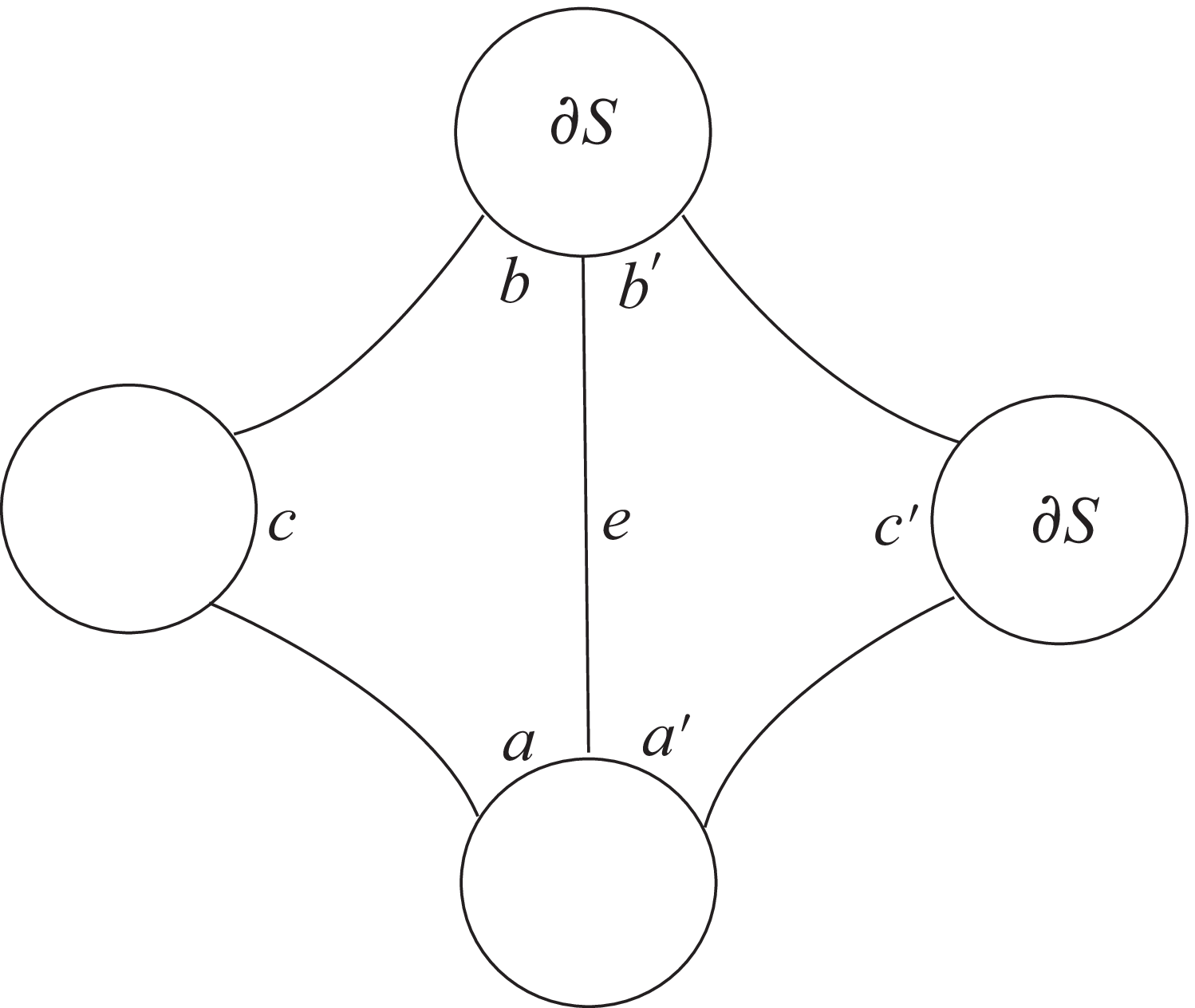}
\end{center}
\caption{\label{fig:invariant}}
\end{figure}

Let $E$ be the set of edges in $T$. For any real number $h$, the $\psi_h$-coordinate of a hyperbolic metric introduced in \cite{l2} is defined as $\psi_h: E \rightarrow \mathbb{R}$,
$$\psi_h(e)=\int_0^{\frac{a+b-c}{2}} \cosh^h(t) dt +
\int_0^{\frac{a'+b'-c'}{2}} \cosh^h(t)dt$$
where $e$ is an
edge shared by two hyperbolic
right-angled hexagons and $c, c'$ are lengths of arcs in the boundary of $S$
facing $e$ and $a, a', b, b'$ are the lengths of arcs in the boundary of $S$ adjacent
to $e$ so that $a, b, c$ lie in a hexagon. See Figure \ref{fig:invariant}.

Now consider the map
$\Psi_h: Teich(S) \to \mathbb{R}^E$ sending a hyperbolic
metric $d$ to its $\psi_h$-coordinate. The following two theorems are proved in \cite{l2}.

\begin{theorem}[\cite{l2}]\label{thm:embedding} Fix an ideal triangulation of $S$.
For each $h \in  \mathbb{R}$, the map $\Psi_h: Teich(S) \to
\mathbb{R}^E$ is a smooth embedding.
\end{theorem}

An \it edge cycle \rm $(e_1, H_1,..., e_n, H_n)$ is a
collection of hexagons and edges in an ideal triangulation so that
two adjacent hexagons $H_{i-1}$ and $H_i$ share the edge $e_i$
for $i=1,...,n$ where $H_0=H_n$. 

\begin{theorem}[\cite{l2}]\label{thm:image} Fix an ideal triangulation of $S$. For each $h\geq 0$, $\Psi_h(Teich(S)) =\{ z \in \mathbb{R}^E | $ for each edge cycle $(e_1, H_1,..., e_n, H_n)$, $\sum_{i=1}^n z(e_i)
>0$\}. Furthermore, the image $\Psi_h(Teich(S))$ is a convex polytope.
\end{theorem}

\section{Hyperbolic right-angled hexagon}

We will use the following notations and conventions.

Given two points $P, Q$ in the hyperbolic plan $\mathbb{H}$, the distance between $P$ and $Q$ will be denoted by $|PQ|$. If $P\neq Q$, the complete geodesic in $\mathbb{H}$ containing $P$ and $Q$ will be denoted by $\overline{PQ}.$ Suppose $H$ is a hyperbolic right-angled hexagon whose vertices are $A_1,B_1,A_2,B_2,A_3,B_3$ labeled cyclically (see Figure \ref{fig:tangent}). Let $C$ be the circle tangent to the three geodesics $\overline{A_1B_1}$, $\overline{A_2B_2}$ and $\overline{A_3B_3}$. The hyperbolic center of $C$ is denoted by $O.$
Let $X_i=C\cap \overline{A_iB_i}$ be the tangent point for $i=1,2,3.$
The geodesic $\overline{B_iA_{i+1}}$ decomposes the hyperbolic plane into two sides. The subindices are counted modulo 3, i.e., $A_4=A_1$ etc.

\begin{lemma}\label{thm:hexagon} The following holds for $i=1,2,3.$
$$|A_iB_i|+|A_{i+1}B_{i+1}|-|A_{i+2}B_{i+2}|=$$
$$\left\{
\begin{array}{rl}
2|X_iB_i|, & \mbox{if $O$ and $H$ are in the same side of $\overline{B_iA_{i+1}}$} \\
0, &  \mbox{if $O\in \overline{B_iA_{i+1}}$} \\
-2|X_iB_i|, & \mbox{if $O$ and $H$ are in different sides of $\overline{B_iA_{i+1}}$}
\end{array}
\right.
$$
\end{lemma}

\begin{proof} Since $X_j$ is the tangent point for $j=1,2,3,$ we have
\begin{equation}\label{fm:1}
|X_jB_j|=|X_{j+1}A_{j+1}|.
\end{equation}
According to the location of $O$ with respect to the hexagon, we have three cases to consider.

\begin{figure}
\begin{center}
\includegraphics[scale=.63]{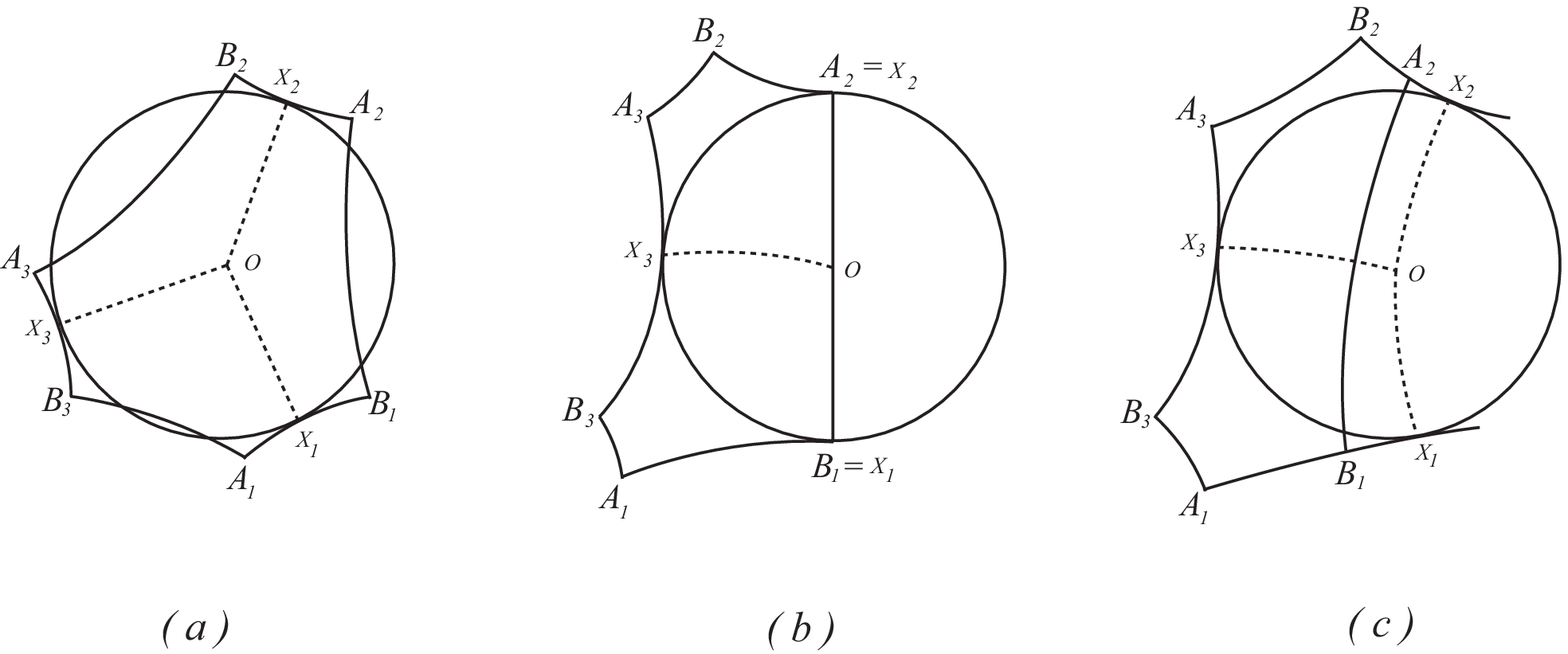}
\end{center}
\caption{\label{fig:tangent}}
\end{figure}

Case 1. If $O$ is in the interior of the hexagon, see Figure \ref{fig:tangent}(a). We have, for $j=1,2,3$,
$$|A_jB_j|=|X_jA_j|+|X_jB_j|.$$
Combining with (\ref{fm:1}), we obtain $|A_jB_j|+|A_{j+1}B_{j+1}|-|A_{j+2}B_{j+2}|=2|X_jB_j|.$ Thus we have verified the lemma in this case since $O$ and $H$ are in the same side of $\overline{B_jA_{j+1}}$ for each $j=1,2,3.$

Case 2. If $O$ is in the boundary of the hexagon, without of losing generality, we assume $O\in \overline{B_1A_2}$. See Figure \ref{fig:tangent}(b). We have
$$|A_1B_1|=|X_1A_1|,$$
$$|A_2B_2|=|X_2B_2|,$$
$$|A_3B_3|=|X_3A_3|+|X_3B_3|.$$
Combining with (\ref{fm:1}), we obtain
$$|A_1B_1|+|A_2B_2|-|A_3B_3|=0,$$
$$|A_2B_2|+|A_3B_3|-|A_1B_1|=2|X_2B_2|,$$
$$|A_3B_3|+|A_1B_1|-|A_2B_2|=2|X_3B_3|.$$
Thus we have verified the lemma in this case since $O\in \overline{B_1A_2}$, $O$ and $H$ are in the same side of $\overline{B_2A_3}$ and in the same side of $\overline{B_3A_1}.$

Case 3. If $O$ is outside of the hexagon $H$, without of losing generality, we may assume $O$ and $H$ are in the same side of $\overline{B_2A_3}$ and in the same side of $\overline{B_3A_1}$, but in different sides of $\overline{B_1A_2}.$ See Figure \ref{fig:tangent}(c). We have
$$|A_1B_1|=|X_1A_1|-|X_1B_1|,$$
$$|A_2B_2|=|X_2B_2|-|X_2A_2|,$$
$$|A_3B_3|=|X_3A_3|+|X_3B_3|.$$
Combining with (\ref{fm:1}), we obtain
$$|A_1B_1|+|A_2B_2|-|A_3B_3|=-2|X_1B_1|,$$
$$|A_2B_2|+|A_3B_3|-|A_1B_1|=2|X_2B_2|,$$
$$|A_3B_3|+|A_1B_1|-|A_2B_2|=2|X_3B_3|.$$
Thus we have verify the lemma in this case.
\end{proof}

\section{Delaunay decompositions}

\begin{figure}
\begin{center}
\includegraphics[scale=.5]{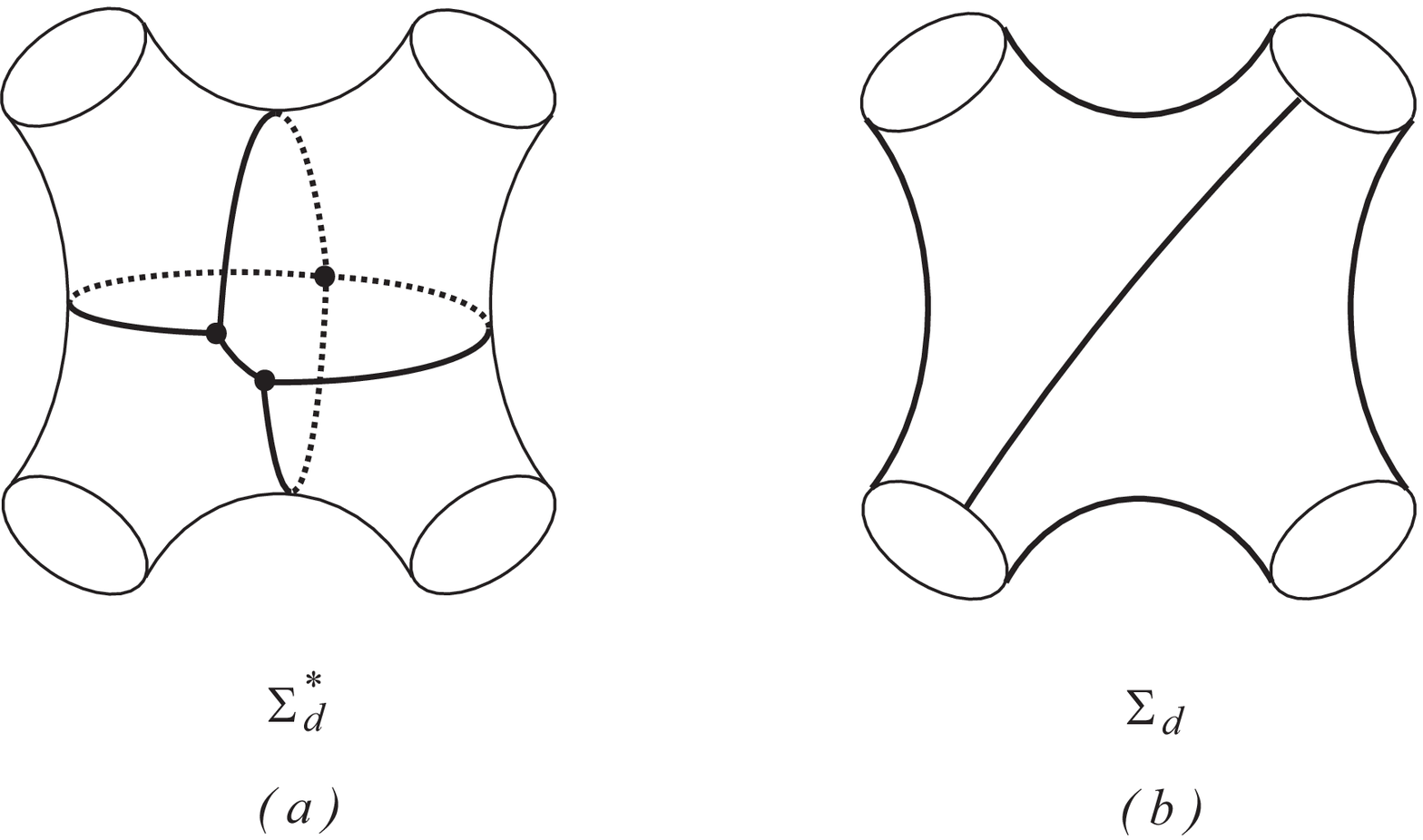}
\end{center}
\caption{\label{fig:spine}}
\end{figure}

Let's recall the construction of the Delaunay decomposition
associated to a hyperbolic metric following Bowditch-Epstein
\cite{be}. For higher dimensional hyperbolic manifolds, see
Epstein-Penner \cite{ep} and Kojima \cite{k}.

Let $(S, d)$ be a hyperbolic metric with geodesic boundary on the
compact surface $S$. Let $d$ be a
hyperbolic metric with geodesic boundary on $S$.  The Delaunay
decomposition of $(S, d)$ produces a graph $\Sigma_d^*$, called
the \it spine \rm of the surface $S$ so that $\Sigma_d^*$ is the
set of points in $S$ which have two or more distinct shortest
geodesics to $\partial S$.

To be more precise, let $n(p)$ be
the number of shortest geodesics arcs from $p$ to $\partial S$. The spine
$\Sigma^*_d$ of $(S,d)$ is the set $\{p\in S|n(p)\geq 2\}$. And
the vertex of $\Sigma^*_d$ is the set $\{p\in S|n(p)\geq 3\}.$ The
set $\Sigma^*_d$ is shown (see Bowditch-Epstein \cite{be}, Kojima \cite{k}) to be a graph whose edges are geodesic arcs in $S$. The the edges of $\Sigma^*_d$ are
denoted by $e^*_1,...,e^*_N.$ By the construction, each of point
in the interior of an edge $e^*_i, i=1,...,N,$ has precisely two distinct shortest
geodesics to $\partial S$. Each edge $e^*_i$ connects the two
vertices which are the points having three or more distinct
shortest geodesics to $\partial S$. By \cite{k} or \cite{be}, it is known that
$\Sigma^*_d$ is a strong deformation retract of the surface $S$.

Associated with the spine $\Sigma_d^*$ is the so called \textit{Delaunay
decomposition} of the hyperbolic surface. Here is the
construction.

For each edge $e^*$ of the spine, there are two boundary components
$B_1$ and $B_2$ (may be coincide) of the surface so that points in the
interior of $e^*$ have exact two shortest geodesic arcs $a_1$ and $a_2$ to
$B_1$ and $B_2$.  Let $e$ be the shortest geodesic from $B_1$ to $B_2$. It is
known that $e$ is homotopic to $a_1 \cup a_2$ and $e$ intersects $e^*$
perpendicularly. Furthermore, these edges $e$'s are pairwise disjoint.
The collection of all such $e$'s decompose the surface $S$ into a
collection of right-angled polygons. These are the 2-cell, or the
Delaunay domains.  We use $\Sigma_d$ to denote the cell decomposition
of the surface $S$ whose 2-cells are the Delaunay domains, whose
1-cells consist of these $e$'s and the arcs in the boundary of $S$.
One can think of $\Sigma_d^*$ as a dual to $\Sigma_d$ as follows. For
each 2-cell $D$ in $\Sigma_d$, there is exactly one vertex $v$ of
$\Sigma_d^*$ so that $v$ lies in the interior of $D$. Furthermore, by the
construction, $v$ is of equal distance to all edges of $D \cap
\partial S$. Consider the hyperbolic circle in $S$ centered at $v$ so
that it is tangent to all edges in $D \cap \partial S$. We call it
the \it inscribed circle \rm of the Delaunay domain $D$.

Figure \ref{fig:spine}(a) is an example of the spine of a four-hole sphere, where the spine
is the graph of thick lines. In Figure \ref{fig:spine}(b), the thick lines produce a Denaulay decomposition.

\section{Proof of the main theorem}

\subsection{Construction of the homeomorphism}

\begin{figure}
\begin{center}
\includegraphics[scale=.5]{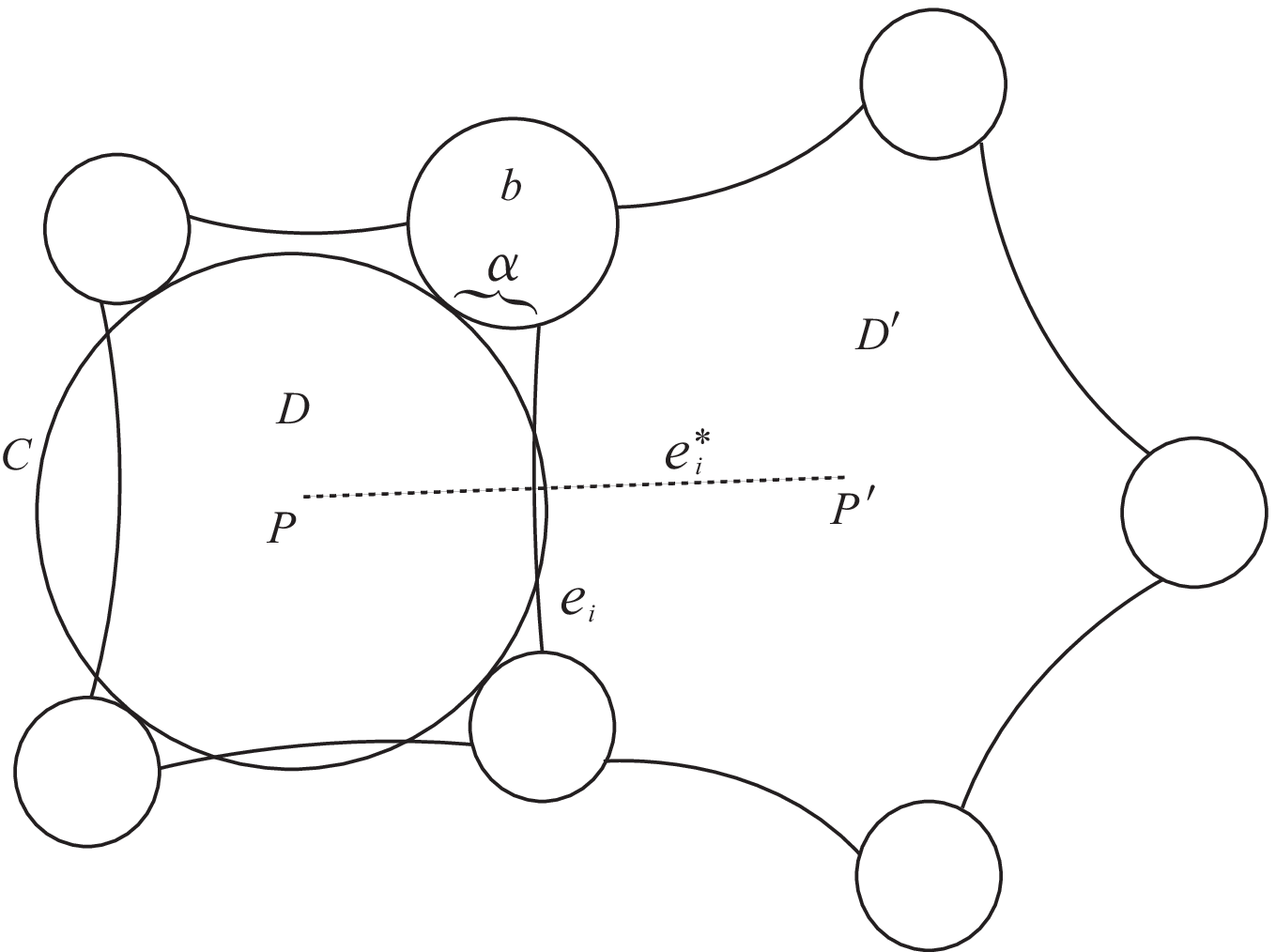}
\end{center}
\caption{\label{fig:definition}}
\end{figure}

To prove Theorem \ref{thm:main}, for each $h \geq 0$, we construct
the map $\Pi_h: Teich(S)\to |A(S)-A_\infty(S)|\times
\mathbb{R}_{>0}$ as follows. Given a hyperbolic metric $d$ with geodesic boundary,
we obtain the spine $\Sigma^*_d$ and the Delaunay
decomposition $\Sigma_d$ of $S$ in the metric $d$. Let
$(e_1^*,...,e_N^*)$ be the edges of the spine and $(e_1,...,e_N)$
be the edges of the Delaunay decomposition where $e_i$ is dual to
$e_i^*.$ See Figure \ref{fig:definition}. Suppose $e_i$ is shared
by two 2-cells $D$ and $D'$. The inscribed circle of $D$ is denoted by $C$. Let $b$ be one of the two edges of $D$ adjacent to the edges $e_i$. Let $\alpha$ be the length of the arc contained in $b$ with end
points $C\cap b$ and $e_i\cap b$. Similarly, we find the inscribed circle of $D'$ and the length $\alpha'$. Now define a function for
each $h\geq 0$:
\begin{equation}\label{fml:pi}
\pi_h(e_i)=\int_0^\alpha \cosh^h(t) dt +
\int_0^{\alpha'} \cosh^h(t)dt.
\end{equation} 
Note that, due to Delaunay
condition,  $\alpha,\alpha'$ are positive for each $i$. Therefore
$\pi_h(e_i)>0$ for each $i$. 

It is clear from the definition that the Delaunay decomposition and the coordinates $\pi_h(e_i)$ depend only on the isotopy class of the hyperbolic metric. In other words, they are independent of the choice of a representative of a point of the Teichm\"uller spaces $Teich(S)$. A point of $Teich(S)$ is denoted by $[d]$. 
We obtain a well-defined map
\begin{align}\label{fml:Pi}
\Pi_h:Teich(S)&\to |A(S)-A_\infty(S)|\times \mathbb{R}_{>0}\\
[d] &\mapsto (\sum_{i=1}^N\frac{\pi_h(e_i)}{\sum_{i=1}^N\pi_h(e_i)}\cdot[e_i], \sum_{i=1}^N\pi_h(e_i)),\notag
\end{align}
where $(e_1,...,e_N)$ are the edges of the Delaunay
decomposition of $(S,d)$ and $[e_i]$ is a isotopy class.
Note that $\sum_{i=1}^N\frac{\pi_h(e_i)}{\sum_{i=1}^N\pi_h(e_i)}\cdot[e_i]$ is
a point in the fillable simplex with vertices $[e_1],...,[e_N]$ of the arc complex,
since the sum of the coefficient of the vertices is 1 and $\pi_h(e_i)>0$ for all $i$.

In the rest of the section, we will show that $\Pi_h$ is injective, onto, and is a homeomorphism.
 
\subsection{One-to-one}

We claim that the map $\Pi_h$ is one-to-one. Suppose there are two hyperbolic
 metrics $d,d'$ such that $\Pi_h([d])=\Pi_h([d'])$.  Then their
 associated Delaunay decompositions are the same by definition.
 Say $\{e_1, ..., e_N\}$ is the set of edges in $\Sigma_d
 =\Sigma_{d'}$. 
If $N=6g-6+3n$ where $g$ is the genus and $n$ is the number of boundary components of $S$, then $(e_1,...,e_N)$ is an ideal triangulation. In this case each 2-cell is a right-angled hexagon. Suppose edge $e_i$ is shared by hexagons $D$ and $D'$. See Figure \ref{fig:triangulation}.

\begin{figure}
\begin{center}
\includegraphics[scale=.35]{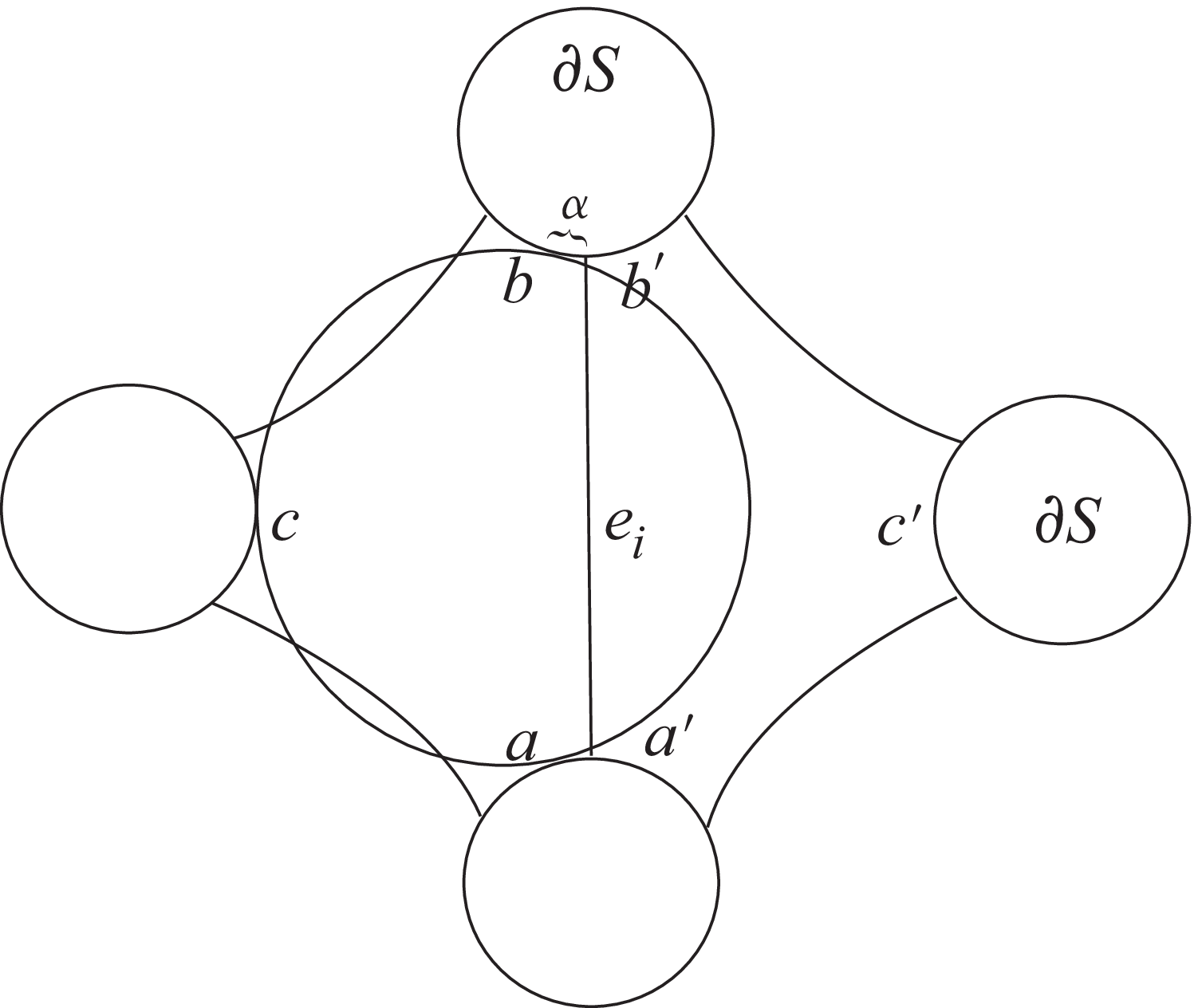}
\end{center}
\caption{\label{fig:triangulation}}
\end{figure}

Let $c$ be the length of boundary arc opposite to $e_i$ and $a,b$ be lengths of boundary arcs adjacent to $e_i$ in $D$. Since the center of the inscribed circle of $D$ and the hexagon $D$ are in the same side of $e_i$, by Lemma \ref{thm:hexagon}, we have $a+b-c=2\alpha$. Similarly, for hexagon $D'$, we have $a'+b'-c'=2\alpha'$. Thus
\begin{align*}
\pi_h(e_i)&=\int_0^\alpha \cosh^h(t) dt + \int_0^{\alpha'} \cosh^h(t)dt\\
&=\int_0^{\frac{a+b-c}{2}} \cosh^h(t) dt +
\int_0^{\frac{a'+b'-c'}{2}} \cosh^h(t)dt\\
&=\psi_h(e_i)
\end{align*}
where $\psi_h(e_i)$ is exactly the $\psi_h$-coordinate of a hyperbolic metric evaluated at $e_i$.
Thus from $\Pi_h([d])=\Pi_h([d'])$ we obtain $\Psi_h([d])=\Psi_h([d'])$ for the ideal triangulation $(e_1,...,e_N)$, $N=6g-6+3n$. By Theorem \ref{thm:embedding}, we see that $[d]=[d']\in Teich(S)$.

If $N<6g-6+3n$, we add edges $e_{N+1},...,e_{6g-6+3n}$ such that $(e_1,...,e_N,e_{N+1},$ $...,e_{6g-6+3n})$ is an ideal triangulation. More precisely, in a 2-cell of the Delaunay decomposition which is not a hexagon, we add arbitrarily geodesic arcs perpendicular to boundary components bounding the 2-cell which decompose the 2-cell into a union of hexagons. 

See Figure \ref{fig:cell}(a). Suppose edge $e_i, i\leq N,$ is shared by two 2-cells $D,D'$. 
There is a hyperbolic right-angled hexagon $H$ contained in $D$ having $e_i$ as an edge. Note that $H$ is a component of $S-\cup_{i=1}^{6g-6+3n}e_i.$ 
Recall that the inscribed circle $C$ of $D$ is also the inscribed circle of $H$. Let $c$ be the length of boundary arc opposite to $e_i$ and $a,b$ be lengths of boundary arcs adjacent to $e_i$ in $H$. Since the center of $C$ and $H$ are in the same side of $e_i$, by Lemma \ref{thm:hexagon}, we have $a+b-c=2\alpha,$ where $\alpha$ is the length in the definition of $\pi_h(e_i).$ From the 2-cell $D'$, we obtain hexagon $H'$ and $a'+b'-c'=2\alpha'.$ Therefore
\begin{align*}
\pi_h(e_i)&=\int_0^\alpha \cosh^h(t) dt + \int_0^{\alpha'} \cosh^h(t)dt\\
&=\int_0^{\frac{a+b-c}{2}} \cosh^h(t) dt +
\int_0^{\frac{a'+b'-c'}{2}} \cosh^h(t)dt\\
&=\psi_h(e_i)
\end{align*}

\begin{figure}
\begin{center}
\includegraphics[scale=.6]{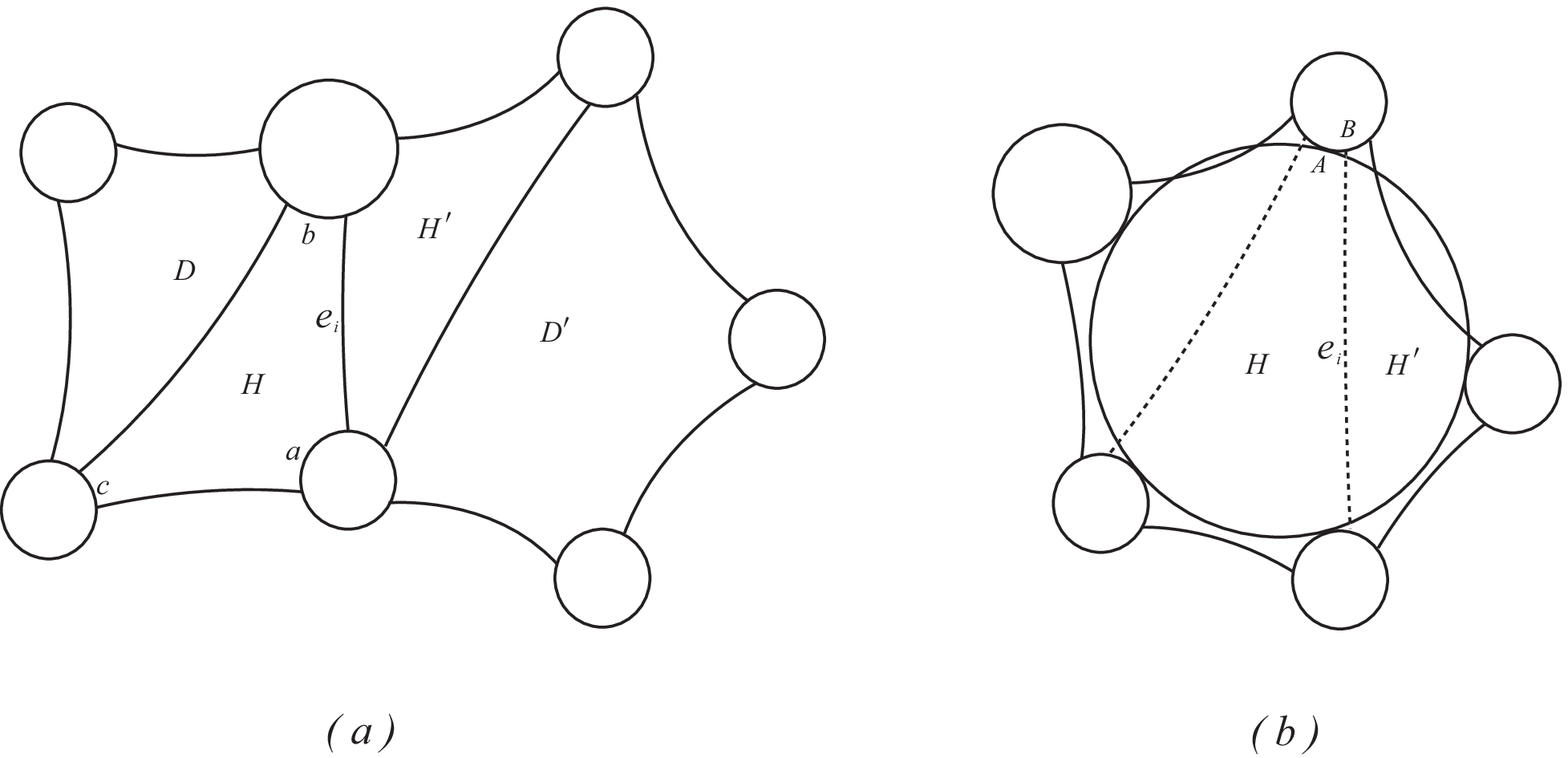}
\end{center}
\caption{\label{fig:cell}}
\end{figure}

See Figure \ref{fig:cell}(b). Suppose edge $e_i, i>N,$ is shared by two hexagon $H,H'$ in the ideal triangulation $(e_1,...,e_N,e_{N+1},...,e_{6g-6+3n})$, where $H$ and $H'$ are obtained from the same 2-cell. Therefore $H$ and $H'$ have the same inscribed circle $C$ which is also the inscribed circle of the 2-cell containing $H,H'.$ In hexagon $H$, let $c$ be the length of boundary arc opposite to $e_i$ and $a,b$ be lengths of boundary arcs adjacent to $e_i$. In hexagon $H'$, we define $a',b',c'$. There are two possibilities to consider. If the center of $C$ is in $e_i.$ Then by Lemma \ref{thm:hexagon}, $a+b-c=0=a'+b'-c'$. If the center of $C$ is not in $e_i,$ without of losing generality, we assume the center and $H$ are in the same side of $e_i.$ Denote by $A$ the tangent point of $C$ at a boundary component. Denote by $B$ the intersection point of $e_i$ with the same boundary component. By Lemma \ref{thm:hexagon}, we have $a+b-c=2|AB|$ and $a'+b'-c'=-2|AB|.$ The two possibilities give the same conclusion, for $i>N$,
$$\psi_h(e_i)=\int_0^x \cosh^h(t) dt +
\int_0^{-x} \cosh^h(t)dt=0.$$

Thus from $\Pi_h([d])=\Pi_h([d'])$ we obtain $\Psi_h([d])=\Psi_h([d'])$ for the ideal triangulation $(e_1,...,e_N,e_{N+1},...,e_{6g-6+3n})$. In fact the $i$-th entry of $\Psi_h([d])=\Psi_h([d'])$ is zero as $N+1\leq i \leq 6g-6+3n$. By Theorem \ref{thm:embedding}, we see that $[d]=[d']\in Teich(S)$.

\subsection{Onto}

We claim the map $\Pi_h:Teich(S)\to |A(S)-A_\infty(S)|\times \mathbb{R}_{>0}$ is onto. Given a point $(\sum_{i=1}^Nz_i\cdot[e_i],x).$ If $N=6g-6+3n,$ then $(e_1,...,e_N)$ is an ideal triangulation of $S$. The vector $(xz_1,...,xz_N)$ satisfies the condition in Theorem \ref{thm:image} since each entry is positive. By Theorem \ref{thm:image}, there is a hyperbolic metric $d$ whose $\psi_h$-coordinate is $(xz_1,...,xz_N)$, i.e., $\psi_h(e_i)=xz_i.$ Since we have shown in last subsection that $\pi_h(e_i)=\psi_h(e_i)$ in this case. Therefore $\Pi_h([d])=(\sum_{i=1}^Nz_i\cdot[e_i],x)$.

If $N<6g-6+3n,$ then $e_1,...,e_N$ is a cell decomposition of $S$. Let $T$ be an ideal triangulation $(e_1,...,e_N,e_{N+1},...,e_{6g-6+3n})$ obtained from the cell decomposition. Then the vector $(xz_1,...,xz_N,0,...,0)$ (there are $6g-6+3n-N$ zeros) satisfies the condition in Theorem \ref{thm:image} since there does not exists an edge cycle consisting of only the ``new" edges $e_i,i>N.$ By Theorem \ref{thm:image}, there is a hyperbolic metric $d$ whose $\psi_h$-coordinate is $(xz_1,...,xz_N,0,...,0)$, i.e., $\psi_h(e_i)=xz_i,i\leq N$ and $\psi_h(e_i)=0, i>N$.

Suppose edge $e_i, i>N$ is shared by two hexagons $H,H'$. By the discussion of last subsection, from $\psi_h(e_i)=0$ we conclude that the inscribed circles of $H$ and $H'$ have the same tangent points at the two boundary components intersecting $e_i$. Therefore the two circles have the same center. Thus they coincide. If a 2-cell is decomposed into several hexagons, then the inscribed circles of all the hexagons coincide. This shows that the 2-cell has a inscribed circle. Thus the cell decomposition $(e_1,...,e_N)$ is the Delaunay decomposition of $(S,h).$

For edge $e_i,i\leq N,$ from the discussion of last subsection, we see $\pi_h(e_i)=\psi_h(e_i).$ Therefore $\Pi_h([d])=(\sum_{i=1}^Nz_i\cdot[e_i],x)$.

\subsection{Continuity of $\Pi_h$}

We follow the idea in \S8 and \S9 of Bowditch-Epstein \cite{be} to prove the continuity. 

Let $\{d^s\}_{s=1}^{\infty}$ be a sequence of hyperbolic metrics on $S$ with geodesic boundary converging to a hyperbolic metric $d$ with geodesic boundary. We claim that the sequence of points $\{\Pi_h([d^s])\}_{s=1}^{\infty}$ converges to the point $\Pi_h([d])$. 

Case 1. If, for $s$ sufficiently large, the Delaunay decomposition associated to $d$ has the same combinatorial type as the Delaunay decomposition associated to $d^s$. Assume that the Delaunay decomposition associated to $d$ has the edges  $e_1,...,e_N$ with $N\leq 6g-6+3n$ and the Delaunay decomposition associated to $d^s$ has the edges $e_1^s,...,e_N^s$ so that $e_i^s$ is isotopic to $e_i$ for $1\leq i \leq N.$ Since the metrics $\{d^s\}$ converge to the metric $d$, the geodesic length of edges $\{e_i^s\}$ converge to the geodesic length of the edge $e_i$. 

Assume that the edge $e_i$ is shared by two 2-cells $D$ and $D'$ of $(S,d)$. Correspondingly, the edge $e_i^s$ is shared by two 2-cells $D^s$ and $D'^s$ of $(S,d^s)$. As in \S 5.1 and Figure \ref{fig:definition}, let $C$ be the inscribed circle of $D$ and $b$ be one of the two edges of $D$ adjacent to $e_i$. Let $\alpha$ be the length of the arc contained in $b$ with end points $C\cap b$ and $e_i\cap b$. Let $\alpha^s$ be the length of the corresponding arc in $D^s$. Assume $e_{D1},...,e_{Dt}\in \{e_1,...,e_N\}$ are the edges of $D$ in the interior of $S$. By the elementary hyperbolic geometry, the radius of $C$ is a continuous function of the lengths of $e_{D1},...,e_{Dt}$. Therefore $\alpha$ is a continuous function of the lengths of $e_{D1},...,e_{Dt}$. Thus the sequence $\{\alpha^s\}$ converges to $\alpha$. By the same argument, for the 2-cell $D'$, we have the length $\alpha'$ and $\alpha'^s$ so that the sequence $\{\alpha'^s\}$ converges to $\alpha'$. By the definition (\ref{fml:pi}), the sequence $\{\pi_h(e_i^s)\}$ converges to $\pi_h(e_i)$. By the definition (\ref{fml:Pi}), the sequence of points $\{\Pi_h([d_s])\}_{s=1}^{\infty}$ converges to the point $\Pi_h([d])$. Geometrically, this is a sequence of interior points in a simplex of the arc complex converging to an interior point in the same simplex. 

Case 2.  If for $s$ sufficiently large, the Delaunay decomposition associated to $d^s$ have the same combinatorial type with each other but different from that associated to $d$. Assume that the Delaunay decomposition associated to $d$ has the edges $e_1,...,e_N$ with $N< 6g-6+3n$ and the Delaunay decomposition associated to $d^s$ has the edges $e_1^s,...,e_N^s,e_{N+1}^s,...,e_{N+M}^s$ with $N+M\leq 6g-6+3n$ so that $e_i^s$ is isotopic to $e_i$ for $1\leq i \leq N.$

Since $e_j^s$ is isotopic to $e_j^{s'}$ for $N+1\leq j \leq N+M$ and $s,s'$ sufficiently large, we can add an edge $e_j$ on $(S,d)$ which is isotopic to $e_j^s$ for $N+1\leq j \leq N+M$. Now the edges $e_1,...,e_N,e_{N+1},...,e_{N+M}$ produce a cell decomposition of $S$ which has the same combinatorial type with the cell decomposition obtained from the edges $e_1^s,...,e_N^s,e_{N+1}^s,...,e_{N+M}^s$. 

We get the same situation of Case 1. The convergence of metrics implies the convergence of the edge lengths which implies the convergence of the $\pi_h$-coordinates. In Case 2, since the edges $e_{N+1},...,e_{N+M}$ are added to a Delaunay decomposition, we know from \S 5.2 that $\pi_h(e_j)=0$ as $N+1\leq j \leq N+M$. Geometrically, this is a sequence of interior points in the a simplex of the arc complex converging to a point on the boundary of the simplex.  

\subsection{Continuity of $\Pi_h^{-1}$}

Let $\{p^s\}_{s=1}^{\infty}$ be a sequence of points in $|A(S)-A_\infty(S)|\times \mathbb{R}_{>0}$ converging to a point $p$. We claim that the sequence of hyperbolic metrics $\{\Pi_h^{-1}(p^s)\}$ converges to the hyperbolic metric $\Pi_h^{-1}(p)$. 

Case 1. If, for $s$ sufficiently large, $\{p^s\}$ and $p$ are in the same simplex, then the Delaunay decomposition associated to $\Pi_h^{-1}(p^s)$ and $\Pi_h^{-1}(p)$ have the same combinatorial type. If it is needed, by adding edges in the 2-cells which are not hexagons, we obtain a fixed topological ideal triangulation of the surface $S$. Note that the $\pi_h(e)=0$ if $e$ is an edge being added. For an edge $e_i$ on $(S,\Pi_h^{-1}(p)),$ denote by $e_i^s$ the corresponding edge on $(S,\Pi_h^{-1}(p^s))$.  
Now we have a fixed ideal triangulation and that the sequence of coordinates $\{\pi_h(e_i^s)\}$ converges to the coordinate $\pi_h(e_i)$ for each edge $e_i$. By \S 5.2, $\pi_h(e_i^s)=\psi_h(e_i^s)$ and  $\pi_h(e_i)=\psi_h(e_i)$. Therefore the sequence of coordinates $\{\psi_h(e_i^s)\}$ converges to the coordinate $\psi_h(e_i)$ for each edge $e_i$. By Theorem \ref{thm:embedding}, the sequence of hyperbolic metrics $\{\Pi_h^{-1}(p^s)\}$ converges to the hyperbolic metric $\Pi_h^{-1}(p)$. 

Case 2. If, for $s$ sufficiently large, $\{p^s\}$ are in the interior of a simplex and $p$ is on the boundary of the simplex. Assume that the Delaunay decomposition associated to $\Pi_h^{-1}(p)$ has the edges $e_1,...,e_N$ with $N< 6g-6+3n$ and the Delaunay decomposition associated to $\Pi_h^{-1}(p^s)$ has the edges $e_1^s,...,e_N^s,e_{N+1}^s,...,e_{N+M}^s$ with $N+M\leq 6g-6+3n$ so that $e_i^s$ is isotopic to $e_i$ for $1\leq i \leq N.$ We can add an edge $e_j$ on $(S,\Pi_h^{-1}(p))$ which is isotopic to the edge $e_j^s$ for $N+1\leq j \leq N+M$. By the assumption that $\{\pi_h(e_i^s)\}$ converges to $\pi_h(e_i)$ for $1\leq i \leq N$ and $\{\pi_h(e_j^s)\}$ converges to $0$ for $N+1\leq j \leq N+M$. Since $e_j$ is added to the Delaunay decomposition of $\Pi_h^{-1}(p)$, $\pi_h(e_j)=0$ as $N+1\leq j \leq N+M$. We get the situation of Case 1. We may add more edges to obtain a fixed ideal triangulation. The same arguments of Case 1 can be used to establish the claim. 

To sum up, we have proved Theorem \ref{thm:main}:
$$\Pi_h: Teich(S)\to |A(S)-A_\infty(S)|\times \mathbb{R}_{>0}$$ is a homeomorphism.

\end{document}